\documentclass[a4paper, 11 pt]{article}

\usepackage{amsfonts,amsmath,amssymb,hhline,ifthen,amsthm,graphics,cite,latexsym, caption2}
\usepackage{latexsym}
\usepackage[T2A]{fontenc}
\usepackage[cp1251]{inputenc}
\usepackage[english]{babel}

\newcounter{q3}
\newcommand{\dsm}[3]{
{\if#20{\if#31{\frac{\partial #1}{\partial y}}\else
          {\frac{\partial^{#3} #1}{\partial y^{#3}}}
        \fi}\else
  {\if#30{\if#21{\frac{\partial #1}{\partial x}}\else
            {\frac{\partial^{#2} #1}{\partial x^{#2}}}
          \fi}\else
    {\setcounter{q3}{#2}\addtocounter{q3}{#3}
    \frac{\partial{\if{1}\arabic{q3}^{}\else^{ \arabic{q3} }\fi}#1}
    {{\if#20\else{\partial x{\if#21\else{^{#2}}\fi}}\fi}
     {\if#30\else{\partial y\if#31\else{^{#3}}\fi}\fi} }}
   \fi}
\fi} }

\newcommand{\dzm}[2]{
{\if#10{\if#21{\frac{\partial}{\partial\overline{\zeta}}}\else
          {\frac{\partial^{#2}}{\partial\overline{\zeta}^{#2}}}
        \fi}\else
  {\if#20{\if#11{\frac{\partial}{\partial\zeta}}\else
            {\frac{\partial^{#1}}{\partial\zeta^{#1}}}
          \fi}\else
    {\setcounter{q3}{#1}\addtocounter{q3}{#2}
    \frac{\partial{\if{1}\arabic{q3}^{}\else^{ \arabic{q3} }\fi}}
    {{\if#10\else{\partial\zeta{\if#21\else{^{#1}}\fi}}\fi}
     {\if#20\else{\partial\overline{\zeta}\if#21\else{^{#2}}\fi}\fi} }}
   \fi}
\fi} }

\newcounter{q2}

\newtheoremstyle{theor}
  {\medskipamount}
  {\medskipamount}
  {\itshape}
  {}
  {\bfseries}
  {.}
  {.5em}
  {}

\newtheorem{definition}{Definition}[section]
\newtheorem{theorem}[definition]{Theorem}
\newtheorem{lemma}[definition]{Lemma}
\newtheorem{proposition}[definition]{Proposition}
\newtheorem{corollary}[definition]{Corollary}

\theoremstyle{definition}
\newtheorem{remark}[definition]{Remark}

\numberwithin{equation}{section}

\makeindex

\newtheoremstyle{remarks}
  {0mm}
  {0mm}
  {\itshape}
  {}
  {\itshape}
  {.}
  {.5em}
  {}
\makeatletter \@addtoreset{equation}{section} \makeatother

\begin{document}

\subsection*{\center ISOMETRIC EMBEDDINGS OF PRETANGENT SPACES  IN $ E^n$}\begin{center}\textbf{V. Bilet and O. Dovgoshey  } \end{center}
\parshape=5
1cm 14.5cm 1cm 14.5cm 1cm 14.5cm 1cm 14.5cm 1cm 14.5cm \noindent \small {\bf Abstract.}
  We prove some infinitesimal analogs of classical results of Menger, Schoenberg and
Blumenthal giving the existence conditions for isometric embeddings of metric spaces in
the finite-dimensional Euclidean spaces.

\parshape=2
1cm 14.5cm 1cm 14.5cm  \noindent \small {\bf Key words:} metric space, pretangent space,
isometric embedding, infinitesimal geometry of metric spaces, Cayley-Menger determinant.

 \bigskip
\textbf{ AMS 2010 Subject Classification: 54E35}

\large \section {Introduction }\hspace*{\parindent} The definition of pretangent and
tangent metric spaces to an arbitrary metric space was introduced in \cite{DM} for
studies of generalized differentiation on metric spaces. The development of this theory
requires the understanding of interrelations between the infinitesimal properties of
initial metric space and  geometry of pretangent spaces to this initial.

The necessary and sufficient conditions under which a pretangent space to metric space
is unique  and  a series of interesting examples of metric spaces with unique pretangent
spaces were presented in \cite{DAKM}.  Some conditions under which pretangent spaces are
compact and bounded were found in a recent paper \cite{DAKC}. Criteria of the
ultrametricity of pretangent spaces were obtained in \cite{DD2} and \cite{DD1} . The
necessary and sufficient conditions under which subspaces $X$ and $Y$ of metric space
$Z$ have the same pretangent spaces in a point of $X\cap Y$ were obtained in \cite{Dov}.
A criterion of the finitness of pretangent spaces was proved in \cite{DM}.

Our main goal is to search the criteria of the isometric embeddability of pretangent
spaces in the real $n-$dimensional Euclidean space $E^n.$ The second part of our paper
contains the general Transfer Principle, Theorem~\ref{Th1.7}, providing, in some cases,
the "automatic translation" of global properties of pretangent spaces into the limits
relations defened in the initial metric spaces. An immediate consequence of the Transfer
Principle is the Conservation Principle describing some properties of metric spaces
which are invariant under passage to the pretangent spaces. In the third part of the
paper we apply the Transfer Principle to the classical condition of isometric
embeddability of metric spaces in $E^{n}$ obtained by K. Menger and I. Schoenberg. We
reformulate their embedding theorems in a suitable form, see Proposition~\ref{main
prop.} and Proposition~\ref{p2.11} and transfer them to the "infinitesimal" embeddings
theorems \ref{Embed} and \ref{Embed*}. In the fourth part we obtain
Theorem~\ref{main.Blum.} which gives the infinitesimal form of Blumenthal's embedding
theorem. Note that in the last case the Transfer Principle do not seem to be applicable.

\section{Pretangent spaces}\hspace*{\parindent}
For convenience we recall the terminology that will be necessary in future.

Let $(X,d)$ be a metric space and let $p$ be a point of $X.$ Fix some sequence
$\tilde{r}$ of positive real numbers $r_n$ tending to zero. In what follows $\tilde{r}$
will be called a \emph{normalizing sequence}. Let us denote by $\tilde{X}$ the set of
all sequences of points from X.
\begin{definition}\label{def:1.1} Two sequences $\tilde{x}=\{x_n\}_{n\in \mathbb N}$ and $\tilde{y}=\{y_n\}_{n\in \mathbb
N},$ $\tilde{x}, \tilde{y} \in \tilde{X}$ are mutually stable with respect to
$\tilde{r}=\{r_n\}_{n\in \mathbb N}$ if there is a finite limit
\begin{equation}\label{eq1.1}
\lim_{n\to\infty}\frac{d(x_n,y_n)}{r_n}:=\tilde{d}_{\tilde{r}}(\tilde{x},\tilde{y})=\tilde{d}(\tilde{x},\tilde{y}).\end{equation}\end{definition}
We shall say that a family $\tilde{F}\subseteq\tilde{X}$ is \emph{self-stable} (w.r.t.
$\tilde{r}$) if every two $\tilde{x}, \tilde{y} \in \tilde{F}$ are mutually stable. A
family $\tilde{F}\subseteq\tilde{X}$ is \emph{maximal self-stable} if $\tilde{F}$ is
self-stable and for an arbitrary $\tilde{z}\in \tilde{X}$ either $\tilde{z}\in\tilde{F}$
or there is $\tilde{x}\in\tilde{F}$ such that $\tilde{x}$ and $\tilde{z}$ are not
mutually stable.

The standart application of Zorn's Lemma leads to the following
\begin{proposition}\label{prop:1.1}Let $(X,d)$ be a metric space and let $p \in X.$ Then for every normalizing sequence $\tilde{r}=\{r_n\}_{n\in \mathbb N}$ there exists a maximal self-stable family $\tilde{X}_p=\tilde{X}_{p,\tilde{r}}$ such that $\tilde{p}:=$ $=\{p,p,...\}\in\tilde{X}_p.$
\end{proposition}

Note that the condition $\tilde{p}\in\tilde{X}_p$ implies the equality
\begin{equation}\label{eq1.1*}\lim_{n\to\infty}d(x_n,p)=0 \end{equation} for every $\tilde{x}=\{x_n\}_{n\in \mathbb N}\in\tilde{X}_p.$

Consider a function $\tilde{d}:\tilde{X}_p\times\tilde{X}_p\rightarrow\mathbb R$ where
$\tilde{d}(\tilde{x},\tilde{y})=\tilde{d}_{\tilde{r}}(\tilde{x},\tilde{y})$ is defined
by \eqref{eq1.1}. Obviously, $\tilde{d}$ is symmetric and nonnegative. Moreover, the
triangle inequality for $d$ implies
$$\tilde{d}(\tilde{x},\tilde{y})\leq\tilde{d}(\tilde{x},\tilde{z})+\tilde{d}(\tilde{z},\tilde{y})$$
for all $\tilde{x},\tilde{y},\tilde{z}$ from $\tilde{X}_p.$ Hence
$(\tilde{X}_p,\tilde{d})$ is a pseudometric space.
\begin{definition}\label{def:1.2} The pretangent space to the space X (at the point p w.r.t. $\tilde{r}$) is the metric identification of the pseudometric space
$(\tilde{X}_{p,\tilde{r}},\tilde{d}).$\end{definition}

Since the notion of pretangent space is important for the present paper, we remind this
metric identification construction.

Define the relation $\sim$ on $\tilde X_p$ by $\tilde x\sim \tilde y$ if and only if
$\tilde d(\tilde x, \tilde y)=0.$ Then $\sim$ is an equivalence relation. Let us denote
by $\Omega_{p,\tilde r}^{X}$ the set of equivalence classes in $\tilde X_p$ under the
equivalence relation $\sim.$ It follows from general properties of pseudometric spaces,
see for example, \cite{Kelley}, that if $\rho$ is defined on $\Omega_{p,\tilde r}^{X}$
by \begin{equation} \label{eq1.2}\rho(\alpha,\beta):=\tilde d (\tilde x, \tilde
y)\end{equation}for $\tilde x\in \alpha$ and $\tilde y\in \beta,$ then $\rho$ is a
well-defined metric on $\Omega_{p,\tilde r}^{X}.$ By definition the metric
identification of $(\tilde X_p, \tilde d)$ is the metric space $(\Omega_{p,\tilde
r}^{X}, \rho).$

Remark that $\Omega_{p,\tilde r}^{X}\ne \varnothing$ because the constant sequence
$\tilde p$ belongs to $\tilde{X}_{p,\tilde{r}},$ see Proposition~\ref{prop:1.1}.

Let $\{n_k\}_{k\in\mathbb N}$ be an infinite, strictly increasing sequence of natural
numbers. Let us denote by $\tilde r'$ the subsequence $\{r_{n_k}\}_{k\in \mathbb N}$ of
the normalizing sequence $\tilde r=\{r_n\}_{n\in\mathbb N}$ and let $\tilde x':=$
$=\{x_{n_k}\}_{k\in\mathbb N}$ for every $\tilde x=\{x_n\}_{n\in\mathbb N}\in\tilde X.$
It is clear that if $\tilde x$ and $\tilde y$ are mutually stable w.r.t. $\tilde r,$
then $\tilde x'$ and $\tilde y'$ are mutually stable w.r.t. $\tilde r'$ and
\begin{equation*}\tilde d_{\tilde r}(\tilde x, \tilde y)=\tilde d_{\tilde r'}(\tilde x', \tilde
y').\end{equation*} If $\tilde X_{p,\tilde r}$ is a maximal self-stable (w.r.t $\tilde
r$) family, then, by Zorn's Lemma, there exists a maximal self-stable (w.r.t $\tilde
r'$) family $\tilde X_{p,\tilde r'}$ such that $$\{\tilde x':\tilde x \in \tilde
X_{p,\tilde r}\}\subseteq \tilde X_{p,\tilde r'}.$$ Denote by $in_{\tilde r'}$ the map
from $\tilde X_{p,\tilde r}$ to $\tilde X_{p,\tilde r'}$ with $in_{\tilde r'}(\tilde
x)=\tilde x'$ for all $\tilde x\in\tilde X_{p,\tilde r}.$ It follows from \eqref{eq1.2}
that after metric identifications $in_{\tilde r'}$ pass to an isometric embedding
$em':\Omega_{p,\tilde r}^{X}~\rightarrow~\Omega_{p,\tilde r'}^{X}$ under which the
diagram
\begin{equation} \label{eq1.3}
\begin{array}{ccc}
\tilde X_{p, \tilde r} & \xrightarrow{\ \ \mbox{\emph{in}}_{\tilde r'}\ \ } &
\tilde X_{p, \tilde r^{\prime}} \\
\!\! \!\! \!\! \!\! \! \pi\Bigg\downarrow &  & \! \!\Bigg\downarrow \pi^{\prime}
\\
\Omega_{p, \tilde r}^{X} & \xrightarrow{\ \ \mbox{\emph{em}}'\ \ \ } & \Omega_{p, \tilde
r^{\prime}}^{X}
\end{array}
\end{equation}is commutative. Here $\pi$ and
$\pi'$ are the natural projections, $\pi(\tilde x):=\{\tilde y \in \tilde X_{p,\tilde
r}: \tilde d_{\tilde r}(\tilde x, \tilde y)=0\}$ and $\pi'(\tilde x):=\{\tilde y \in
\tilde X_{p,\tilde r'}: \tilde d_{\tilde r'}(\tilde x, \tilde y)=0\}.$

Let $X$ and $Y$ be two metric spaces. Recall that the map $f:X\rightarrow Y$ is called
an \emph{isometry} if $f$ is distance-preserving and onto.

\begin{definition}\label{def:1.3}A pretangent $\Omega_{p,\tilde
r}^{X}$ is tangent if $em':\Omega_{p,\tilde r}^{X}\rightarrow \Omega_{p,\tilde r'}^{X}$
is an isometry for every~$\tilde r'.$\end{definition}

\begin{remark}\label{rem 1.1} Let $\tilde X_{p,\tilde r}$ be a maximal self-stable family with corresponding pretangent space $\Omega_{p,\tilde r}^{X}.$ Then  $\Omega_{p,\tilde r}^{X}$ is tangent if and only if for every subsequence $\tilde r'=\{r_{n_k}\}_{k\in\mathbb N}$ of the sequence $\tilde r$ the family $\tilde X_{p, \tilde r'}:=\{\tilde x': \tilde x \in \tilde X_{p,\tilde r}\}$ is maximal self-stable w.r.t. $\tilde r'.$\end{remark}
 For every natural $k \ge 1$ write $X^{k+1}$ for the set of all $k+1-$tuples
$x=(x_0,x_1,...,x_k)$ with terms $x_n \in X$ for $n=0,1,...,k.$

Denote by $\textbf{\emph{M}}_{n}, n\in\mathbb N,$ the topological space of all real,
$n\times n-$matrices $\textbf{\emph{t}}$ with the topology of pointwise convergence. Let
$\mathfrak{M}$ be a class of nonvoid metric spaces and let $\mathfrak{F}$ be a family of
continuous functions $f: \textbf{\emph{M}}_{n}\rightarrow \mathbb R, n=n(f)$ which are
homogeneous of degree $s_{0}=s_{0}(f)>0,$ i.e.,
\begin{equation}\label{e1.5}
f(\delta \textbf{\emph{t}})=\delta^{s_{0}}(f(\textbf{\emph{t}}))
\end{equation}
for all $\delta\in [0, \infty)$ and all $\textbf{\emph{t}}\in Dom(f).$ We shall say that
$\mathfrak{M}$ is determined by $\mathfrak{F}$ if the following two conditions are
equivalent for every metric space $(X, d):$ $(X, d)\in\mathfrak{M};$  the inequality
$f(\textbf{\emph{m}})\ge 0$ holds for each $f\in\mathfrak{F}$ and all
$\textbf{\emph{m}}\in Dom(f)$ having the form
\begin{equation}\label{e1.6}
\textbf{\emph{m}}=\textbf{\emph{m}}(x_1, x_2,...,x_n)=\left(\begin{array}{cccc}
   d(x_{1}, x_{1})&d(x_{1}, x_{2})&...&d(x_{1}, x_{n})\\
   d(x_{2}, x_{1})&d(x_{2}, x_{2})&...&d(x_{2}, x_{n})\\
    \vdots&\vdots&\ddots&\vdots\\
   d(x_{n}, x_{1})&d(x_{n}, x_{2})&...&d(x_{n}, x_{n})\\
    \end{array}\right), (x_1, x_2, ..., x_n)\in X^{n}.
\end{equation}
\begin{remark}\label{rem 1.2} Equality \eqref{e1.5} and the inequality $s_{0}(f)>0$ imply that $f(\textbf{0})=0$ for every $f\in \mathfrak{F}$
where $\textbf{0}$ is the zero $n\times n-$matrix belonging to $Dom(f).$ It is clear
that each matrix \eqref{e1.6} is equal to $\textbf{0}$ for one-point metric spaces.
 Consequently each one-point metric space belongs to every $\mathfrak{M}$
determinated by some $\mathfrak{F}.$
\end{remark}
For example, the class of all ultrametric spaces is determinated by the family
$\mathfrak{F}$ with the unique element $f: \textbf{\emph{M}}_{3}\rightarrow \mathbb R,$
\begin{equation*}
f(\textbf{\emph{t}})=(t_{1,3} \vee t_{3,2}) -t_{1,2}.
\end{equation*}
Indeed, if $\textbf{\emph{t}}$ has form $\eqref{e1.6},$ then the inequality
$f(\textbf{\emph{t}})\ge 0$ can be written as the ultra-triangle inequality $d(x_{1},
x_{2})\le d(x_{1}, x_{3})\vee d(x_{3}, x_{2}).$

 Let $(X,d)$ be a metric space with marked point $p$ and let $f\in\mathfrak{F}.$ We
set
\begin{equation}\label{e1.8}
\delta(x_1,...,x_n):=\mathop{\vee}\limits_{i=1} ^ {n}d(x_{i},p)
\end{equation} for $(x_1,...,x_n)\in X^{n}$ and define the function $f^{*}: X^{n}\rightarrow \mathbb R$ as
\begin{equation}\label{Func}
f^{*}(x_1,x_2,...,x_n):=\begin{cases}
         f\left(\frac{\textbf{\emph{m}}(x_1, x_2,...,x_n)}{\delta(x_1, x_2,...,x_n)} \right) & \mbox{if} $ $ (x_1,x_2,...,x_n) \ne (p,p,...,p)\\
         0& \mbox{if}$ $ (x_1,x_2,...,x_n) = (p,p,...,p). \\
         \end{cases}
\end{equation}
\begin{theorem}\label{Th1.7}\textbf{(}\textbf{Transfer Principle}\textbf{)}
Let $(X,d)$ be a metric space with marked point p and let $\mathfrak{M}$ be a family of
metric spaces determinated by a family $\mathfrak{F}.$ The following two statements are
equivalent.
\begin{enumerate}
\item[\rm(i)]\textit{Each pretangent space $\Omega_{p,\tilde r}^X$ belongs to $\mathfrak{M}.$}

\item[\rm(ii)]\textit{The inequality \begin{equation}\label{eq1.10} \liminf_{x_1,x_2,...,x_n \to p}f^{*}(x_1,x_2,...,x_n)\ge 0 \end{equation} holds for each $f: \textbf{M}_{n}\rightarrow \mathbb R$ belonging to $\mathfrak{F}.$}
\end{enumerate}

\end{theorem}
\begin{proof}
Suppose that (i) holds. Let us prove inequality \eqref{eq1.10}  for each
$f\in\mathfrak{F}.$ Let $f: \textbf{M}_{n}\rightarrow \mathbb R$ belong to
$\mathfrak{F}$ and let $\tilde x_{i} =\{x_{i,m}\}_{m\in\mathbb N}\in\tilde X,$
$i=1,2,...,n,$ be some sequences such that
\begin{equation}\label{eq1.11}
\lim_{m\to\infty}f^{*}(x_{1,m},...,x_{n,m})=\liminf_{x_1,...,x_n \to
p}f^{*}(x_1,...,x_n)
\end{equation}
and
\begin{equation}\label{eq1.12}
p=\lim_{m\to\infty}x_{1,m}=\lim_{m\to\infty}x_{2,m}=...=\lim_{m\to\infty}x_{n,m}.
\end{equation}
Limit relation \eqref{eq1.12} implies
$$\lim_{m\to\infty}\delta (x_{1,m},...,x_{n,m})=0$$ where $\delta$ is defined by
\eqref{e1.8}. If for all sufficiently large $m$ we have $\delta(x_{1,m},...,x_{n,m})=0,$
then the limit in \eqref{eq1.11} vanishes, so that \eqref{eq1.10} holds. Consequently we
may suppose, going to a subsequence, that $$\delta(x_{1,m},...,x_{n,m})>0$$ for all
$m\in\mathbb N.$ Define a normilizing sequence $\tilde r=\{r_{m}\}_{m\in\mathbb N}$ as
$$r_{m}:=\delta(x_{1,m},...,x_{n,m}), \, m\in\mathbb{N}.$$
All elements of the matrix $\frac{\textbf{\emph{m}}(x_1,...,x_n)}{\delta(x_1,...,x_n)},$
see \eqref{e1.6}, are bounded because
\begin{equation}\label{e2.9*}0\le\frac{\mathop{\vee}\limits_{i,j=1} ^ {n}d(x_{i,m},x_{j,m})}{r_m} \le
\frac{2\mathop{\vee}\limits_{i=1} ^ {n}d(x_{i,m},p)}{r_m}=2. \end{equation} Hence going
to a subsequence once again we can assume that all $\tilde x_{i},$ $i=1,...,n,$ and
$\tilde p$ are pairwise mutually stable. The functions $f\in\mathfrak{F}$ are
continuous. Hence using \eqref{Func} we obtain
\begin{equation}\label{eq1.13}
\lim_{m\to\infty}f^{*}(x_{1,m},...,x_{n,m})=f(\emph{\textbf{t}}),
\end{equation} where
\begin{equation*}\emph{\textbf{t}}=
\left(\begin{array}{cccc}
   \tilde d(\tilde x_{1}, \tilde x_{1})&\tilde d(\tilde x_{1}, \tilde x_{2})&...&\tilde d(\tilde x_{1},\tilde x_{n})\\
   \tilde d(\tilde x_{2}, \tilde x_{1})&\tilde d(\tilde x_{2}, \tilde x_{2})&...&\tilde d(\tilde x_{2},\tilde x_{n})\\
    \vdots&\vdots&\ddots&\vdots\\
   \tilde d(\tilde x_{n}, \tilde x_{1})&\tilde d(\tilde x_{n}, \tilde x_{2})&...&\tilde d(\tilde x_{n}, \tilde x_{n})\\
    \end{array}\right).
\end{equation*}
If $\tilde X_{p, \tilde r}$ is a maximal self-stable family such that $\tilde
x_{i}\in\tilde X_{p,\tilde r}, $ $i=1,...,n$ and $\Omega_{p,\tilde r}^{X}$ is the metric
identification of $\tilde X_{p, \tilde r},$ then $\Omega_{p,\tilde r}^{X}
\in\mathfrak{M}.$ Since the family $\mathfrak{M}$ is determined by $\mathfrak{F}$ and
\begin{equation*}\emph{\textbf{t}}=
\left(\begin{array}{cccc}
 \rho(\alpha_{1}, \alpha_{1})&\rho (\alpha_{1}, \alpha_{2})&...&\rho(\alpha_{1},\alpha_{n})\\
   \rho(\alpha_{2}, \alpha_{1})&\rho(\alpha_{2}, \alpha_{2})&...&\rho(\alpha_{2},\alpha_{n})\\
    \vdots&\vdots&\ddots&\vdots\\
   \rho(\alpha_{n}, \alpha_{1})&\rho(\alpha_{n}, \alpha_{2})&...&\rho(\alpha_{n}, \alpha_{n})\\
    \end{array}\right)
\end{equation*}
where $\alpha_{i}=\pi(\tilde x_{i}),$ see \eqref{eq1.3}, we obtain the inequality
$$f(\emph{\textbf{t}})\ge 0.$$ This inequality, \eqref{eq1.13}and \eqref{eq1.11} imply
\eqref{eq1.10}.

Assume now that \eqref{eq1.10} holds for all $f\in\mathfrak{F}.$ We must prove that each
$\Omega_{p,\tilde r}^{X}$ belongs to $\mathfrak{M}.$ Let $\Omega_{p,\tilde r}^{X}$ be a
pretangent space with corresponding maximal self-stable family $\tilde X_{p,\tilde r}.$
The relation $\Omega_{p,\tilde r}^{X} \in \mathfrak{M}$ means that for every $f:
\textbf{\emph{M}}_{n}\rightarrow \mathbb R$ the inequality
\begin{equation}\label{eq1.14}
f(\emph{\textbf{m}}(\alpha_1, ..., \alpha_n))\ge 0
\end{equation} holds for all $\alpha_1, ..., \alpha_n \in \Omega_{p,\tilde r}^{X}$ where
\begin{equation*}\emph{\textbf{m}}(\alpha_1, ..., \alpha_n)=
\left(\begin{array}{cccc}
 \rho(\alpha_{1}, \alpha_{1})&\rho (\alpha_{1}, \alpha_{2})&...&\rho(\alpha_{1},\alpha_{n})\\
   \rho(\alpha_{2}, \alpha_{1})&\rho(\alpha_{2}, \alpha_{2})&...&\rho(\alpha_{2},\alpha_{n})\\
    \vdots&\vdots&\ddots&\vdots\\
   \rho(\alpha_{n}, \alpha_{1})&\rho(\alpha_{n}, \alpha_{2})&...&\rho(\alpha_{n}, \alpha_{n})\\
    \end{array}\right).
\end{equation*}
Inequality \eqref{eq1.14} holds automatically if $$\mathop{\vee}\limits_{i,j=1} ^
{n}\rho(\alpha_i, \alpha_j)=0,$$ see Remark~\ref{rem 1.2}. Hence we may suppose that
$$\mathop{\vee}\limits_{i,j=1} ^ {n}\rho(\alpha_i, \alpha_j)>0.$$ If $\alpha =\pi(\tilde
p),$ then the last inequality implies
\begin{equation}\label{eq1.15}
\mathop{\vee}\limits_{i=1} ^ {n}\rho(\alpha, \alpha_i)>0.
\end{equation}
Let $\tilde x_{i}=\{x_{i,m}\}_{m\in\mathbb N},i~=~1,...,n$ be elements of $\tilde
X_{p,\tilde r}$ such that $\alpha_{i}=\pi(\tilde x_{i}).$  Using inequality
\eqref{eq1.15} we can write
\begin{equation*}
\rho(\alpha_{i},\alpha_{j})=\lim_{m\to\infty}\frac{d(x_{i,m},
x_{j,m})}{r_{m}}=\lim_{m\to\infty}\frac{\delta(x_{1,m},x_{2,m},...,x_{n,m})}{r_{m}}\frac{d(x_{i,m},x_{j,m})}{\delta(x_{1,m},x_{2,m},...,x_{n,m})}
\end{equation*}
\begin{equation}\label{eq1.16}
=\mathop{\vee}\limits_{i=1} ^ {n}\rho(\alpha, \alpha_i)\lim_{m\to\infty}\frac{d(x_{i,m},
x_{j,m})}{\delta(x_{1,m},x_{2,m},...,x_{n,m})}
\end{equation} for $i,j=1,...,n.$ From \eqref{e1.5}, \eqref{Func} and
\eqref{eq1.16} we obtain
\begin{equation*}
f\left( \frac{\emph{\textbf{m}}(\alpha_{1},...,\alpha_{n})}{\mathop{\vee}\limits_{i=1} ^
{n}\rho(\alpha, \alpha_i)}\right)=\lim_{m\to\infty}f^{*}(x_{1,m},x_{2,m},...,x_{n,m}),
\end{equation*}
\begin{equation}\label{eq1.17}
f(\emph{\textbf{m}}(\alpha_{1},...,\alpha_{n}))=\left(\mathop{\vee}\limits_{i=1} ^
{n}\rho(\alpha,
\alpha_i)\right)^{s_{0}}\lim_{m\to\infty}f^{*}(x_{1,m},x_{2,m},...,x_{n,m})
\end{equation}
where $s_{0}>0$ is the degree of homogeneity of f. Since
\begin{equation*}
\lim_{m\to\infty}f^{*}(x_{1,m},x_{2,m},...,x_{n,m})\ge \liminf_{x_{1},x_{2},...,x_{n}
\to p}f^{*}(x_1,...,x_n)\ge 0
\end{equation*} and $$\left(\mathop{\vee}\limits_{i=1} ^ {n}\rho(\alpha, \alpha_i) \right)^{s_{0}}>0,$$
equality \eqref{eq1.17} implies \eqref{eq1.14}.
\end{proof}
Let $f:\textbf{\emph{M}}_{n}\rightarrow \mathbb R$ be a continuous homogeneous function
with the degree of homogeneity $s_{0}=s_{0}(f)>0,$ let $(X,d)$ be a metric space with a
marked point $p$ and let $f^{*}: X^{n}\rightarrow \mathbb R$ be the function given by
\eqref{Func}. Define the family $\mathfrak{U}$ of metric space $(X,d)$ by the rule
\begin{equation}\label{Fam}
(X,d) \in \mathfrak{U} \Leftrightarrow f(\textbf{\emph{m}}(x_1,...,x_n ))=0
\end{equation}
for all $(x_1,...,x_n)\in X^{n}$ where $\textbf{\emph{m}}(x_1,...,x_n)$ is the matrix of
form \eqref{e1.6}.
\begin{corollary}\label{col1.8}
Let $(X,d)$ be a metric space with a marked point $p.$ All pretangent spaces
$\Omega_{p,\tilde r}^X$ belong to $\mathfrak{U}$ if and only if
\begin{equation}\label{e1.18}
\lim_{x_1,...,x_n \to p}f^{*}(x_1,...,x_n)=0.
\end{equation}
\end{corollary}
\begin{proof}
Let us consider the two-point set $\mathfrak{F}=\{f, -f\}.$ Note that $(-f)$ is also
continuous homogeneous function of degree $s_{0}.$ The family $\mathfrak{U}$ is
determined by $\mathfrak{F}$ because $f(\textbf{\emph{m}}(x_1,...x_n))~=~0$ if and only
if $f(\textbf{\emph{m}}(x_1,...x_n))\ge 0$ and $-f(\textbf{\emph{m}}(x_1,...x_n))\ge 0.$
Hence, by Theorem~\ref{Th1.7}, all pretangent spaces $\Omega_{p,\tilde r}^X$ belong to
$\mathfrak{U}$ if and only if
\begin{equation}\label{e1.19}
\liminf_{x_1,...,x_n \to p}f^{*}(x_1,...x_n)\ge 0 \quad \mbox{and} \quad
\liminf_{x_1,...,x_n \to p}(-f^{*}(x_1,...x_n))\ge 0.
\end{equation}
The last inequality is the equivalent of
\begin{equation*}
\limsup_{x_1,...,x_n \to p}f^{*}(x_1,...x_n)\le 0 .
\end{equation*}
This inequality and the first inequality in \eqref{e1.19} give \eqref{e1.18}.
\end{proof}
\begin{remark}\label{rem1.9}
The proof of Theorem~\ref{Th1.7} is a generalization of the proof of Theorem~3.1 from
\cite{BD} which gives the necessary and sufficient conditions under which all pretangent
spaces $\Omega_{p,\tilde r}^{X}$ are \emph{ptolemaic}. These conditions lead to a
criterion of isometric embeddability of pretangent spaces in $E^{1}.$
\end{remark}
The following corollary is of interest in its own right.
\begin{corollary}\textbf{(Conservation Principle)} \label{col1.10}
Let $\mathfrak{M}$ be a class of nonvoid metric spaces determined by a family
$\mathfrak{F}.$ Then for every metric space $X\in\mathfrak{M}$
 all pretangent spaces $\Omega_{p,\tilde r}^{X}$ belong to $\mathfrak{M}$ for each $p\in X.$\end{corollary}
\begin{remark}\label{rem1.11}
It is plain to prove that in the Transfer Principle instead of the function
\begin{equation*}\delta(x_1,...,x_n)=\mathop{\vee}\limits_{i=1} ^
{n}d(x_{i},p)
\end{equation*} we can use an arbitrary function $\varepsilon: X^{n}\rightarrow
[0,\infty)$ fulfilling the restrictions $$\varepsilon (x_1,...,x_n)=0\Leftrightarrow
x_1=...=x_n=p$$ and
\begin{equation*}\frac{1}{c}\le\liminf_{x_1,...,x_n\to p}\frac{\varepsilon (x_1,...,x_n)}{\delta(x_1,...,x_n)}\le\limsup_{x_1,...,x_n\to p}\frac{\varepsilon
(x_1,...,x_n)}{\delta(x_1,...,x_n)}\le c
\end{equation*} with some constant $c\in [1,\infty).$ Here we put
\begin{equation*}
\frac{\varepsilon (p,...,p)}{\delta(p,...,p)}=1.
\end{equation*}
For example we can take
\begin{equation*}\varepsilon (x_1,...,x_n)=\left(\sum_{i=1}^{n}d^{s}(p,x_i)\right)^{\frac{1}{s}}
\end{equation*} with $s>0.$
\end{remark}

\section {Infinitesimal versions of Menger's and Shoenberg's embedding theorems }
\hspace*{\parindent} In this section we start with the reformulation of the Menger
Embedding Theorem in a suitable form  for application of the Transfer Principle. Recall
that the Cayley-Menger determinant is the next determinant

$$D_{k}(x_0,x_1,...,x_k)=\begin{vmatrix}
    0&1&1&...&1\\
    1&0&d^{2}(x_0,x_1)&...&d^{2}(x_0,x_k)\\
    1&d^{2}(x_1,x_0)&0&...&d^{2}(x_1,x_k)\\
    \vdots&\vdots&\vdots&\ddots&\vdots\\
    1&d^{2}(x_k,x_0)&d^{2}(x_k,x_1)&...&0\\
    \end{vmatrix}$$ where $(x_0,x_1,...,x_k)\in X^{k+1}.$

\begin{proposition}\label{main prop.} Let $n \in \mathbb N.$ A metric space $X$ is isometrically embeddable in $E^n$ if and only if
\begin{equation}\label{eq2.1} (-1)^{k+1}D_{k}(x_0,x_1,...,x_k)\ge 0 \end{equation} for every $(x_0,x_1,...,x_k)\in X^{k+1}$ with $k\le n$ and \begin{equation}\label{eq2.2} D_{k}(x_0,x_1,...,x_k)=0\end{equation} for every $(x_0,x_1,...,x_k)\in X^{k+1}$ with $k=n+1$ and $k=n+2$.\end{proposition}
To prove Proposition \ref{main prop.} we shall use some known results of K. Menger and
L. Blumenthal. Our first lemma is the simplest form of the Menger Embedding Theorem.
\begin{lemma}\label{Menger}A metric space $X$ is isometrically embeddable in $E^n$ if and only if each set $A\subseteq X$ with\emph{ card}$A \le n+3$ is isometrically embeddable in $E^n.$  \end{lemma}
The clear proof of it can be found in [\cite{Bl}, p.95].

The following lemma is a corollary of Blumenthal's solution of the problem of isometric
embedding of semimetric spaces in the Euclidean spaces, see [\cite{Bl}, p.105].

\begin{lemma}\label{Blum}Let $X$ be a finite metric space with \emph{ card}$X = n+1.$ Then $X$ is isometrically embeddable in $E^n$ if and only if the Cayley-Menger determinant $D(x_0,x_1,...,x_k)$ has the sign of $(-1)^{k+1}$ or vanishes for every  $(x_0, x_1,...,x_k)\in X^{k+1},$ $k=1,2,...,n.$ \end{lemma}

\begin{proof}[Proof of Proposition \ref{main prop.}.] Suppose that $X$ is isometrically
embeddable in $E^{n}.$ Let $(x_0, x_1, ..., x_k)\in$ $\in X^{k+1}.$ If $k\le n,$ then
inequality \eqref{eq2.1} follows directly from Lemma \ref{Blum}. Let $k=n+1$ or $k=n+2.$
We can consider $E^n$ as a subspace of the Euclidean space $E^k.$

Let $F$ be an isometric embedding of $X$ in $E^k.$

Write $x_{0}^{*}:=F(x_0), x_{1}^{*}:=F(x_1), ..., x_{k}^{*}:=F(x_k)$ and denote by
$V(x_{0}^{*},x_{1}^{*},...,x_{k}^{*})$ the volume of the simplex with vertices
$x_{0}^{*},x_{1}^{*},...,x_{k}^{*}.$ This simplex lies in the subspace $E^n$ of the
space $E^k.$ Thus, we have \begin{equation}\label{eq2.3}
V(x_{0}^{*},x_{1}^{*},...,x_{k}^{*})=0. \end{equation} Since
\begin{equation}\label{eq2.4}
V^{2}(x_{0}^{*},x_{1}^{*},...,x_{k}^{*})=\frac{(-1)^{k+1}}{2^{k}(k!)^{2}}D_{k}(x_{0}^{*},x_{1}^{*},...,x_{k}^{*})=\frac{(-1)^{k+1}}{2^{k}(k!)^{2}}D_{k}(x_{0},x_{1},...,x_{k}),
\end{equation} see, for example, [\cite{Bl}, p.98], these equalities and \eqref{eq2.3}
imply \eqref{eq2.2}.

Consequently, suppose for every $(x_0,x_1,...,x_k)\in X^{k+1}$ we have \eqref{eq2.1} if
$k\le n$ or \eqref{eq2.2} if $k=n+1,\, k=n+2.$ We must show that $X$ is isometrically
embeddable in $E^n.$ By Lemma \ref{Menger}  it is sufficient to prove that every
$A\subseteq X$ with card$A\le n+3$ has this property. Note that it follows directly from
Lemma \ref{Blum} if card$A\le n+1.$

Let us consider the case where $$A=\{x_0, x_1,...,x_n,x_{n+1},x_{n+2}\}$$ (the case
$A=\{x_0,...,x_n,x_{n+1}\}$ is more simple and can be considered similarly). By Lemma
\ref{Blum}, there is an isometric embedding $F: A \rightarrow E^{n+2},\,$
$F(x_0)=x_{0}^{*}, ..., F(x_{n+2})=x_{n+2}^{*}.$ We may assume, without loss of
generality, that $x_{0}^{*}=0.$ Denote by $L$ the linear subspace of $E^{n+2}$ generated
by the vectors $x_{1}^{*}, ..., x_{n}^{*}, x_{n+1}^{*}, x_{n+2}^{*}.$ It is clear that
$A$ is isometrically embeddable in $E^{n}$ if $dimL\le n.$ If the last inequality does
not hold, then the set $\{x_{1}^{*},...,x_{n}^{*},x_{n+1}^{*},x_{n+2}^{*}\}$ contains
some linear independent vectors  $x_{1}^{**},...,x_{n}^{**},x_{n+1}^{**}.$

Let $x_{1}^{'},...,x_{n}^{'},x_{n+1}^{'}$ be elements of the set $\{
x_1,...,x_n,x_{n+1},x_{n+2}\}$ such that $$x_1^{**}=F(x_{1}^{'}),\,
x_2^{**}=F(x_{2}^{'}),..., x_{n+1}^{**}=F(x_{n+1}^{'}).$$ Since $x_1^{**}, x_2^{**},...,
x_{n+1}^{**}$ are linear independent, we have $$V(x_0^{*}, x_1^{**},...,
x_{n+1}^{**})>0.$$ Using the last inequality and \eqref{eq2.4} we obtain $$D_{k}(x_0,
x_1^{'},..., x_n^{'},x_{n+1}^{'})\ne 0,$$ contrary to equality \eqref{eq2.2}.
\end{proof}

Let $(X,d)$ be a metric space with a marked point p. Similarly $\eqref{Func}$ define the
functions $\Theta_{k+1}:X^{k+1}\rightarrow \mathbb R$ by the rule
\begin{equation}\label{Teta}
\Theta_{k+1}(x_0,x_1,...,x_k):=\begin{cases}
         \frac{(-1)^{k+1}D_{k}(x_0,x_1,..., x_k)}{(\mathop{\vee}\limits_{n=0} ^
{k}d(x_n,p))^{2k}}, & \mbox{if} $ $ (x_0,x_1,...,x_k) \ne (p,p,...,p)\\
         0,& \mbox{if}$ $ (x_0,x_1,...,x_k) = (p,p,...,p) \\
         \end{cases}
\end{equation} where $\mathop{\vee}\limits_{n=0} ^
{k}d(x_n,p):=\mathop{\max}\limits_{0\le n \le k } d(x_n,p).$

The following theorem gives necessary and sufficient conditions under which all
pretangent spaces have isometric embeddings in $E^n.$

\begin{theorem}\label{Embed}Let $(X,d)$ be a metric space with a marked point p and let $n \in \mathbb N.$
Every $\Omega_{p,\tilde r}^X$ is isometrically embeddable in $E^{n}$ if and only if
inequality \begin{equation}\label{pn} \liminf_{x_0,x_1,...,x_k \to
p}\Theta_{k+1}(x_0,x_1,...,x_k)\ge 0 \end{equation} holds for all $k\le n$ and the
equality \begin{equation}\label{pv} \lim_{x_0,x_1,...,x_k \to
p}\Theta_{k+1}(x_0,x_1,...,x_k)= 0 \end{equation} holds for  $k=n+1$ and $k=n+2.$
\end{theorem}
The theorem can be proved by application of Theorem~\ref{Th1.7} and
Corollary~\ref{col1.8} with $\mathfrak{M}$ equals the class of all metric spaces which
are embeddable in $E^{n}$ and $$\mathfrak{F}=~\{D_1,D_2,..., D_{n}\}\cup
\{D_{n+1},-D_{n+1}, D_{n+2}, -D_{n+2}\}.$$ Note only that all Cayley-Menger determinants
$D_1,...,D_n, D_{n+1}$ and $D_{n+2}$ are continuous functions on
$\textbf{\emph{M}}_{2},...,\textbf{\emph{M}}_{n+1}, \textbf{\emph{M}}_{n+2}$ and
$\textbf{\emph{M}}_{n+3}$  and with degrees of homogeneity equal $2,..., 2n, 2(n+1),
2(n+2)$ respectively.

\begin{remark}\label{K-M} The main information about Cayley-Menger determinants can be found in the books of M. Berger \cite{Berger1} and L. Blumenthal \cite{Bl}. These determinants play
an important role in some questions of metric geometry. In 1928 Menger used them to
characterize the Euclidean spaces solely in metric terms. They also participate in
metric characterization of Riemann's manifolds of the constant sectional curvature,
obtained by Berger \cite{Berger2}. In a recent paper \cite{DC}, it was proved that the
Cayley-Menger determinant of an $n-$dimensional simplex is an absolutely irreducible for
$n \ge 3.$ The following results,indicated also in \cite{DC}, are found in using of
these determinants: this is a proof of the bellows conjecture, which asserts that all
flexible polyhedra keep a constant volume in 3-dimensional Euclidean space (see,
\cite{CSW}, \cite{S}); the study of the spatial form of the molecules in the
stereochemistry \cite{KB}.
\end{remark}
The following is immediate from the Conservation Principle.
\begin{corollary}\label{C2.6}If $X$ is a subset of $ E^{n}$ and $p\in X,$ then all pretangent spaces
$\Omega_{p, \tilde r}^{X} $ are isometrically embeddable in $E^n.$\end{corollary}

Let $X$ be a metric space with a marked point $p.$ Define the second pretangent space to
$X$ at the point $p\in X$ as a pretangent space to a pretangent space $\Omega_{p, \tilde
r}^{X}.$ More generally suppose we have constructed all $n-$th pretangent spaces to $X$
at $p.$ We shall denote such spaces as $\Omega^{n}=\left(\Omega^{n}, \rho_{n}\right).$
\begin{definition}\label{def2.7}
A metric space $Y$ is an $(n+1)-$th pretangent space to $X$ at $p$ if there are an
$n-$th pretangent space $\left(\Omega^{n}, \rho_{n}\right)$ and a point
$p_{n}\in\Omega^{n}$ and a normilizing sequence $\tilde r^{n}$ and a maximal self-stable
family $\tilde \Omega_{p_{n},\tilde r_{n}}^{n}\subseteq\tilde\Omega^{n}$ such that $Y$
is the metric identification of the pseudometric space $\left(\tilde
\Omega_{p_{n},\tilde r_{n}}^{n}, \tilde \rho_{n} \right).$
\end{definition}

\begin{corollary}\label{C2.8} Let $X$ be a metric space with a marked point $p$ and let $k\in\mathbb N.$ If each (first) pretangent space to $X$ at $p$ is isometrically embeddable in $E^{k},$ then for every $n\ge 2$ all $n-$th pretangent spaces to $X$  at $p$ are  also isometrically embeddable in $E^k.$ \end{corollary}

Let $a_{jk},\, 0\le j,k\le n$ be real constants such that $a_{jj}=0$ and $a_{jk}=a_{kj}$
if $k \ne j.$  The following is Schoenberg's embedding theorem from \cite{Sch}.
\begin{theorem}\label{Schoen}
A necessary and sufficient condition that $a_{jk}$ be the lengths of the adges of an
$n-$simplex lying in $E^{m},$ but not in $E^{l}$ with $l<m$ is that the quadratic form
\begin{equation}\label{eq2.14} F(y_1, y_2, ..., y_n)=\sum_{j,k=1}^{n}(a_{0j}^{2}+a_{0k}^{2}-a_{jk}^{2})y_{j}y_{k}\end{equation}
be positive semidefinite and of rank m.
\end{theorem}
For applications of this result to embeddings of pretangent spaces in $E^{m}$ we
introduce the determinant $Shc(x_0, x_1,...,x_n)$. Let $(X,d)$ be a metric space and let
$(x_0,x_1,...,x_n)\in X^{n+1}.$ Write
\begin{equation} \label{DetSc}
Shc(x_0, x_1, ..., x_n)=\begin{vmatrix}
    \tau_{11}&\tau_{12}&...&\tau_{1n}\\
    \tau_{21}&\tau_{22}&...&\tau_{2n}\\
    \vdots&\vdots&\vdots&\vdots\\
    \tau_{n1}&\tau_{n2}&...&\tau_{nn}\\
    \end{vmatrix}
\end{equation} where
\begin{equation}\label{eq2.16}\tau_{ij}=d^{2}(x_{0},x_{i})+d^{2}(x_{0},x_{j})-d^{2}(x_{i},x_{j})\end{equation} for $1\le i,j
\le n.$
\begin{lemma}\label{lemSch}
Let $X$ be a finite metric space with card$X$=$n+1, \, n\ge 1.$ Then $X$ is
isometrically embeddable in $E^{n}$ if and only if the inequality
\begin{equation}\label{Sch} Sch(x_0, x_1,...,x_k)\ge 0 \end{equation} holds for every
$(x_0,x_1,...,x_k)\in X^{k+1}, k=1,2,...,n.$
\end{lemma}
\begin{proof}
Suppose that $X$ is isometrically embeddable in $E^{n}.$ The determinant
$Sch(x_0,x_1,...,x_k)$ vanishes if there is $i_0\in[1,...,k]$ with $x_{i_0}=x_0$ or
there are distinct $i_0, j_0 \in [1,...,k]$ such that $x_{i_0}=x_{j_0}.$ Indeed, in the
first case we have
\begin{equation*}
\tau_{i_{0}j}=d^{2}(x_{0},x_{0})+d^{2}(x_{0},x_{j})-d^{2}(x_{0},x_{j})=0
\end{equation*}
for all $j\in[1,...,k].$ Similarly if $k\ge 2$ and $x_{i_0}=x_{j_0}$ we obtain
\begin{equation*}
\tau_{i_{0}j}=d^{2}(x_{0},x_{i_{0}})+d^{2}(x_{0},x_{j})-d^{2}(x_{i_{0}},x_{j})=
d^{2}(x_{0},x_{j_{0}})+d^{2}(x_{0},x_{j})-d^{2}(x_{j_{0}},x_{j})=\tau_{j_{0}j}
\end{equation*}
for all $j\in[1,k].$

Hence it is sufficient to show \eqref{Sch} if all $x_{i}, \, i\in [0,...,n]$ are
pairwise distinct. Since $X$ is isometrically embeddable in $E^{n},$
Theorem~\ref{Schoen} implies that quadratic form \eqref{eq2.14} is positive
semidefinite. A well-known criterion states that a quadratic form is positive
semidefinite if and only if all principal minors of the matrix of this form are
nonegative, see, for example,\cite[p.~272]{Gant}. Hence \eqref{Sch} follows.

Conversely, suppose that inequality $\eqref{Sch}$ holds for all $(x_0,x_1,..,x_k)\in
X^{k+1}, \, k=1,...,n.$ The criterion given above, implies that quadratic form
\eqref{eq2.14} is positive semidefinite. Let us denote by $m$ the rank of this form. It
is clear that $m\le n.$ Consequantly there is an isometric embedding of $X$ in $E^{m}$
and thus in $E^{n}$ also.
\end{proof}

The next proposition is similar to Proposition~\ref{main prop.}.
\begin{proposition}\label{p2.11}
Let $n\in\mathbb N$ and let $(X,d)$ be a nonvoid metric space. The metric space $(X,d)$
is isometrically embeddable in $E^{n}$ if and only if the inequality
\begin{equation}\label{eq2.18}
Sch(x_0,x_1,...,x_k)\ge 0
\end{equation}
holds for every $(x_0,x_1,...,x_k)\in X^{k+1}$ with $k=1,...,n$ and the equality
\begin{equation}\label{eq2.19}
Sch(x_0,x_1,...,x_k)= 0
\end{equation}
holds for every $(x_0,x_1,...,x_k)\in X^{k+1}$ with $k=n+1, n+2.$
\end{proposition}
\begin{proof}
If \eqref{eq2.18} holds for all $(x_0,x_1,...,x_k)\in X^{k+1},$ $k=1,...,n$ and
\eqref{eq2.19} holds for all $(x_0,x_1,...,x_k)\in X^{k+1},$ $k=n+1, n+2,$ then
quadratic form \eqref{eq2.14} is positive semidefinite and the rank of this form is at
most $n.$ (Recall that the rank of quadratic form is the rank of matrix of this form.)
Consequently if $A$ is a subspace of $X$ and $cardX\le n+2$ then, by Lemma~\ref{lemSch},
$A$ is isometrically embeddable in $E^{n}.$ Now Lemma~\ref{Menger} implies that $X$ is
also isometrically embeddable in $E^{n}.$

It still remains to note that if $X$ is isometrically embeddable in $E^{n},$ then
\eqref{eq2.18} and \eqref{eq2.19} follows directly from Theorem~\ref{Schoen} and
Lemma~\ref{lemSch}.
\end{proof}
Let $(X,d)$ be a metric space with marked point $p.$ Define the function
$S_{k+1}:X^{k+1}\rightarrow \mathbb R$ by analogy with the function $\Theta_{k+1},$ see
\eqref{Teta}.
\begin{equation}\label{S}
S_{k+1}(x_0,x_1,...,x_k):=\begin{cases}
         \frac{Sch(x_0,x_1,...,x_k)}{(\mathop{\vee}\limits_{n=0} ^
{k}d(x_n,p))^{2k}}, & \mbox{if} $ $ (x_0,x_1,...,x_k) \ne (p,p,...,p)\\
         0,& \mbox{if}$ $ (x_0,x_1,...,x_k) = (p,p,...,p). \\
         \end{cases}
\end{equation}
The next theorem is similar to Theorem~\ref{Embed} but it presents some other necessary
and sufficient conditions of isometric embeddability of all pretangent spaces to $X$ at
the marked point $p.$

Applying the Transfer Principle and Corollary~\ref{col1.8} to Proposition~\ref{p2.11} we
obtain the following infinitesimal analog of Schoenberg's Embedding Theorem.
\begin{theorem}\label{Embed*}
Let $(X,d)$ be a metric space with a marked point $p$ and let $n\in\mathbb N.$ Every
$\Omega_{p,\tilde r}^X$ is isometrically embeddable in $E^{n}$ if and only if the
inequality \begin{equation}\label{eq2.21} \liminf_{x_0,x_1,...,x_k \to
p}S_{k+1}(x_0,x_1,...,x_k)\ge 0 \end{equation} holds for all $k\le n$ and the equality
\begin{equation}\label{eq2.22} \liminf_{x_0,x_1,...,x_k \to p}S_{k+1}(x_0,x_1,...,x_k)=0 \end{equation} holds for all $k=n+1$ and $k=n+2.$

\end{theorem}

\section{Application of Blumenthal's embedding theorem}
\hspace*{\parindent} Theorem~\ref{Embed} and Theorem~\ref{Embed*}  proved in the
previous section describe some necessary and sufficient conditions under which
\textbf{all} pretangent spaces $\Omega_{p,\tilde r}^{X}$ are isometrically embeddable in
$E^{n}$ with given $n$ but it is possible that there exists an isometric embedding of
\textbf{a fixed} $\Omega_{p,\tilde r}^{X}$ in $E^{n}$ even if these conditions do not
occur. We study this situation in the present section. It turns out that a suitable tool
for these studies is an infinitesimal modification of Blumenthal's embedding theorem. We
first reformulate this theorem in a appropriate form.
\begin{theorem}\label{Blumenthal} A metric space $(X,d)$ is isometrically embeddable in $E^n, n \ge
1,$ if and only if there are some points $a_0, a_1, ..., a_n \in X$ such that
\begin{equation}\label{eq3.1} (-1)^{k+1}D_{k}(a_0, a_1,...,a_k)>0  \end{equation}for each
$k=1,2,...,n$ and the equalities \begin{equation}\label{eq3.2}
D_{k+1}(a_0,a_1,...,a_n,y)=0,\quad D_{k+2}(a_0,a_1,...,a_n,y,z)=0\end{equation} hold for
all $y,z\in X.$ Moreover if $\eqref{eq3.1}$ holds for $k=1,2,...,n$ and $\eqref{eq3.2}$
holds for all $y, z\in X,$ then there are not isometric embeddings of X in $E^{m}$ with
$m<n.$
\end{theorem}

 The proof of Theorem~\ref{Blumenthal} is a straightforward application of theorems 41.1
 and 42.1 and of Lemma 42.1 from \cite{Bl} to the standart form of Blumenthal's
 embedding  theorem, see Theorem 38.1 in \cite{Bl}, and we omit it here.

\begin{theorem}\label{main.Blum.}
Let $(X,d)$ be a metric space with a marked point $p.$ If there are a \textbf{tangent}
space $\Omega_{p,\tilde r}^{X}$ and a natural number $n$ such that $\Omega_{p,\tilde
r}^{X}$ is isometrically embeddable in $E^{n}$ but there are not isometric embeddings of
this $\Omega_{p,\tilde r}^{X}$ in $E^{l}$ with $l<n,$ then there exist some sequences
$$\tilde x^{i}=\{x_{m}^{i}\}_{m\in\mathbb N}\in\tilde X, \quad i=0,1,...,n,$$
having the following properties:
\item[\rm(i)]\textit{The limit relations \begin{equation}\label{eq.3.4} \lim_{m \to\infty}x_{m}^{0}=\lim_{m
\to\infty}x_{m}^{1}=...=\lim_{m \to\infty}x_{m}^{n}=p\end{equation} and
\begin{equation}\label{eq.3.5} \mathop{\land}\limits_{k=1} ^
{n}\liminf_{m\to\infty}\Theta_{k+1}(x_{m}^{0}, x_{m}^{1},..., x_{m}^{k})>0\end{equation}
hold;}
\item[\rm(ii)]\textit{The equalities \begin{equation}\label{eq.3.6} \lim_{m\to\infty}\Theta_{n+2}(x_{m}^{0}, x_{m}^{1},..., x_{m}^{n},y_m)=0\end{equation} and  \begin{equation}\label{eq.3.7} \lim_{m\to\infty}\Theta_{n+3}(x_{m}^{0}, x_{m}^{1},..., x_{m}^{n},y_m,u_m)=0\end{equation}
hold for $\tilde y=\{y_m\}_{m\in\mathbb N}\in\tilde X, \tilde u=\{u_m\}_{m\in\mathbb
N}\in\tilde X$ if \begin{equation}\label{eq.3.8} \lim_{m \to\infty}u_{m}=\lim_{m
\to\infty}y_{m}=p.\end{equation}}Conversely, suppose that there are $\tilde x^{0},
...,\tilde x^{n}\in \tilde X $ having properties (i)-(ii), then there is a
\textbf{pretangent} space $\Omega_{p,\tilde r}^{X}$ which is isometrically embeddable in
$E^{n}$ but there are not  isometric embeddings of this $\Omega_{p,\tilde r}^{X}$ in
$E^{l}$ with $l<n.$
\end{theorem}
Recall that the functions $\Theta_{k}$ were defined by \eqref{Teta}.
\begin {lemma}\label{lem3.3} Let $(X,d)$ be a metric space with a marked point p, $\mathfrak{B}$  a countable subfamily of $\tilde X, \tilde r=\{r_n\}_{n\in\mathbb N}$ a normalizing sequence and let $\tilde X_{p,\tilde r}$ be a maximal self-stable family.
Suppose that the inequality \begin{equation}\label{eq3.9}
\limsup_{n\to\infty}\frac{d(y_n,p)}{r_n}<\infty\end{equation} holds for every $\tilde
y=\{y_n\}_{n\in\mathbb N}\in\mathfrak{B}$ and that a pretangent space $\Omega_{p,\tilde
r}^{X}=\pi(\tilde X_{p,\tilde r})$ is separable and tangent.Then there is a strictly
increasing, infinite sequence $\{n_k\}_{k\in\mathbb N}$ of natural numbers such that for
every $\tilde y=\{y_n\}_{n\in\mathbb N}\in\mathfrak{B}$ there exists $\tilde
z=\{z_n\}_{n\in\mathbb N}\in\tilde X_{p,\tilde r}$ with $\tilde z^{'}=\tilde y^{'},$
i.e., the equality \begin{equation}\label{e3.10} z_{n_{k}}=y_{n_{k}}\end{equation} holds
for all $k\in\mathbb N.$
\end{lemma}
For the proof see Proposition 3 in \cite{DAKC}.

\begin{proof}[Proof of Theorem~\ref{main.Blum.}.] Suppose that there are a tangent space
$\Omega_{p,\tilde r}^{X}$ and natural $n$ such that $\Omega_{p,\tilde r}^{X}$ is
isometrically embeddable in $E^{n}$ but there not isometric embeddings of
$\Omega_{p,\tilde r}^{X}$ in $E^{l}$ with $l<m.$ By Theorem~\ref{Blumenthal} the metric
space $\Omega_{p,\tilde r}^{X}$ contains some points $\beta_0, \beta_1,..,\beta_n$ such
that
\begin{equation}\label{e3.11}(-1)^{k+1}D_{k}(\beta_0,...,\beta_k)>0\end{equation} for
$k=1,...,n$ and
\begin{equation}\label{e3.12}D_{n+1}(\beta_0,
\beta_1,..,\beta_n,\gamma)=D_{n+2}(\beta_0, \beta_1,..,\beta_n,\gamma,
v)=0\end{equation}for all $\gamma, v \in \Omega_{p,\tilde r}^{X}.$ Let $\tilde X_{p,
\tilde r}$ be a maximal self-stable family corresponding $\Omega_{p,\tilde r}^{X}$ and
let $\tilde x^{i}=\{x_{m}^{i}\}_{m\in\mathbb N},$ $i=0,..,n$ be elements of $\tilde
X_{p,\tilde r}$ such that $\pi(\tilde x^{i})=\beta_{i}, i=0,...,n,$ where $\pi$ is the
natural projection. We claim that these $\tilde x^{0}, ..., \tilde x^{n}$ have
properties (i) and (ii).

To prove it note firstly that \eqref{eq.3.4} follows from \eqref{eq1.1*}. Moreover we
have the equality
\begin{equation*}\lim_{m\to\infty}\frac{1}{r_m}\left(\mathop{\lor}\limits_{i=0} ^ {k}d(x_{m}^{i},p)\right) =\mathop{\lor}\limits_{i=0} ^ {k}\rho(\alpha, \beta_i)\end{equation*}
for $k=1,...,n$ where $\alpha=\pi(\tilde p).$ This equality and \eqref{Teta} imply
\begin{equation*}\mathop{\land}\limits_{k=1} ^ {n}\liminf_{m\to\infty}\Theta_{k+1}(x_{m}^{0},...,x_{m}^{k})
=\mathop{\land}\limits_{k=1} ^ {n}\left(\frac{1}{\left (\mathop{\lor}\limits_{i=0} ^
{k}\rho(\alpha, b_i)\right)^{2k}}
\liminf_{m\to\infty}\frac{(-1)^{k+1}D_{k}(x_{m}^{0},...,x_{m}^{k})}{r_{m}^{2k}}\right)\end{equation*}
\begin{equation}\label{e3.13}=\mathop{\land}\limits_{k=1} ^ {n}\left(\frac{1}{\left (\mathop{\lor}\limits_{i=0} ^ {k}\rho(\alpha, b_i)\right)^{2k}} (-1)^{k+1}D_{k}(\beta_0,...,\beta_k)\right).\end{equation}
It should be pointed here that
\begin{equation}\label{e3.14} \mathop{\lor}\limits_{i=0} ^ {k}\rho(\alpha, \beta_{i})>0\end{equation}
for $k=1,..,n.$ Indeed in the opposite case we have
$\alpha=\beta_{0}=\beta_{1}=...=\beta_{k}$ that implies $D_{k}(\beta_0,
\beta_1,...,\beta_k)=0$ for $k=1,...,n,$ contrary to \eqref{e3.11}. Now using
\eqref{e3.11}, \eqref{e3.13} and \eqref{e3.14} we obtain \eqref{eq.3.5}.

Let us prove property (ii). Let $\tilde y=\{y_m\}_{m\in\mathbb N}\in\tilde X$ be a
sequence such that $$\lim_{m\to\infty}y_{m}=p$$ and let $c$ be a limit point of the
sequence $\{\Theta_{n+2}(x_{m}^{0}, x_{m}^{1},...,x_{m}^{n},y_{m})\}_{m\in\mathbb N},$
i.e.,

\begin{equation} \label{e3.15}
\lim_{k\to\infty}\Theta_{n+2}(x_{m_k}^{0},x_{m_k}^{1},...,x_{m_k}^{n},y_{m_k})=c\end{equation}for
some sequence $\{m_k\}_{k\in\mathbb N}.$ We must prove $c=0.$

Inequalities \eqref{e2.9*} imply that the function $\Theta_{n+2}$ is bounded from above
and from below. Consequently $c$ is finite. For convenience we write
$x_{k}^{1,0}=x_{m_k}^{0},...,x_{k}^{1,n}=x_{m_k}^{n}, \, y_{k}^{1}=y_{m_k}$ and
$r_{k}^{1}=r_{m_k}$ so that we have
\begin{equation}\label{e3.16} \lim_{k\to\infty}\Theta_{n+2}(x_{k}^{1,0},x_{k}^{1,1},...,x_{k}^{1,n}, y_{k}^{1})=c. \end{equation}
Note also that the space $\Omega_{p,\tilde r}^{X}$ is tangent by the condition of the
theorem and separable as an isometrically embeddable in $E^{n}$ space. Furthermore,
according to Remark~\ref{rem 1.1} the family $\tilde X_{p, \tilde r'}=in_{\tilde
r'}(\tilde X_{p, \tilde r}),$ see \eqref{eq1.3}, is maximal self-stable w. r. t. the
normilizing sequence $\tilde r'=\{r_{m_k}\}_{k\in\mathbb N},$ so that we can use
Lemma~\ref{lem3.3}. If the inequality
\begin{equation}\label{e3.16*}
\limsup_{k\to\infty}\frac{d(y_{k}^{1},p)}{r_{k}^{1}}<\infty
\end{equation}
holds, then using this lemma with $\mathfrak{B}$ consisting of the unique element
$\{y_{k}^{1}\}_{k\in\mathbb N}$ we can find $\{z_{k}^{1}\}_{k\in\mathbb N}\in\tilde
X_{p, \tilde r'}$ and strictly increasing infinite sequence $\{k_j\}_{j\in\mathbb N}$ of
natural numbers such that $$y_{k_j}^{1}=z_{k_j}^{1}$$ for all $j\in\mathbb N.$ Using
Remark~\ref{rem 1.1} we see that there is $\tilde z=\{z_m\}_{m\in\mathbb N}\in \tilde
X_{p, \tilde r}$ such that $$z_{m_k}=z_{k}^{1}$$ for all $k\in\mathbb N.$ Write $\gamma
= \pi (\tilde z).$ Similarly \eqref{e3.13} we have
\begin{equation*}c=\lim_{k\to\infty}\Theta_{n+2}(x_{k}^{1,0},x_{k}^{1,1},...,x_{k}^{1,n},y_{k}^{1})=\lim_{j\to\infty}\Theta_{n+2}(x_{k_j}^{1,0},x_{k_j}^{1,1},...,x_{k_j}^{1,n}, z_{k_j}^{1}) =\end{equation*}
\begin{equation}\label{e3.17}\lim_{m\to\infty}\Theta_{n+2}(x_{m}^{0},...,x_{m}^{n},z_{m})=\frac{1}{\left(\left(\mathop{\lor}\limits_{i=0} ^ {n}\rho(\alpha, \beta_{i}) \right)\lor\left(\rho(\alpha,\gamma)\right)\right)^{2(n+1)}} (-1)^{n+2}D_{n+1}(\beta_{0},...,\beta_{n},\gamma). \end{equation}
It follows from \eqref{e3.12} and \eqref{e3.14} that
\begin{equation*}0=\mathrm{sign} D_{n+1}(\beta_{0},...,\beta_{n},\gamma)=(-1)^{(n+2)}\mathrm{sign}\, c. \end{equation*}
Thus $c=0$ if \eqref{e3.16*} holds. Suppose contrary that
\begin{equation*}\limsup_{k\to\infty}\frac{d(y_{k}^{1},p)}{r_{k}^{1}}=\infty.\end{equation*}
Let $\tilde y^{1'}=\{y_{k_j}^{1}\}_{k\in\mathbb N}$ be a subsequence of
$\{y_{k}^{1}\}_{k\in\mathbb N}$ such that
\begin{equation}\label{e3.18}\lim_{j\to\infty}\frac{d(y_{k_j}^{1},p)}{r_{k_j}^{1}}=\infty.\end{equation}
In this case we have
\begin{equation}\label{e3.19}\left(\mathop{\lor}\limits_{i=0} ^ {n}d(x_{k_{j}}^{1,i}, p) \right)\lor\left(d(y_{k_{j}}^{1},p)\right)=d(y_{k_{j}}^{1},p)\end{equation}
for all sufficiently large $j.$ In addition, \eqref{e3.18} and \eqref{e3.19} imply the
limit relations
\begin{equation}\label{e3.20}\lim_{j\to\infty}\frac{d(x_{k_{j}}^{1,s}, x_{k_{j}}^{1,t})}{\left(\mathop{\lor}\limits_{i=0} ^ {n}d(x_{k_{j}}^{1,i}, p) \right)\lor\left(d(y_{k_{j}}^{1},p)\right)}=0,\end{equation}
\begin{equation}\label{e3.21}\lim_{j\to\infty}\frac{d(x_{k_{j}}^{1,t}, y_{k_j}^{1})}{\left(\mathop{\lor}\limits_{i=0} ^ {n}d(x_{k_{j}}^{1,i}, p) \right)\lor\left(d(y_{k_{j}}^{1},p)\right)}=1\end{equation}
for all $s,t\in\{1,2,...,n\}.$ Consequently we have
\begin{equation}\label{e3.22}\lim_{j\to\infty}\Theta_{n+2}(x_{k_j}^{1,0},x_{k_j}^{1,1},...,x_{k_j}^{1,n}, y_{k_j}^{1})
=(-1)^{n+2}\begin{vmatrix}
    0&1&1&...&1&1\\
    1&0&0&...&0&1\\
    1&0&0&...&0&1\\
    \vdots&\vdots&\vdots&...&\vdots&\vdots\\
    1&0&0&...&0&1\\
    1&1&1&...&1&0\\
    \end{vmatrix}.\end{equation}
The second row of this determinant coincides with the third one, thus the determinant is
zero. Hence in \eqref{e3.15} we have $c=0.$

Let us turn to equality \eqref{eq.3.7}. Consider, as in \eqref{e3.16}, two sequences
$\tilde y=\{y_{m}\}_{m\in\mathbb N}$ and $\tilde u=\{u_{m}\}_{m\in\mathbb N}$ such that
\begin{equation*} p=\lim_{m\to\infty}y_{m}=\lim_{m\to\infty}u_{m}\end{equation*}and
\begin{equation}\label{e3.23}\lim_{k\to\infty}\Theta_{n+3}(x_{k}^{1,0},x_{k}^{1,1},..,x_{k}^{1,n},y_{k}^{1},u_{k}^{1})=c\end{equation}
where the constant $c$ is an arbitrary limit number of the sequence
$$\{\Theta_{n+3}(x_{m}^{0},x_{m}^{1},...,x_{m}^{n},y_{m},u_{m})\}_{m\in\mathbb N}.$$ As in
the prove of equality \eqref{eq.3.6} we want to use Lemma~\ref{lem3.3} for the
demonstration of equality $c=0.$ In accordance with this lemma it is relevant to
consider three possible cases:

\begin{equation*}(i_1) \qquad \limsup_{k\to\infty}\frac{d(y_{k}^{1},p)}{r_{k}^{1}}<\infty \quad \mbox{and} \quad
\limsup_{k\to\infty}\frac{d(u_{k}^{1},p)}{r_{k}^{1}}<\infty ;\end{equation*}
\begin{equation*}(i_2) \qquad \limsup_{k\to\infty}\frac{d(y_{k}^{1},p)}{r_{k}^{1}}<\infty \quad \mbox{and} \quad
\limsup_{k\to\infty}\frac{d(u_{k}^{1},p)}{r_{k}^{1}}=\infty \end{equation*} \qquad
\qquad \qquad \qquad \quad or
\begin{equation*} \, \, \, \quad \qquad \limsup_{k\to\infty}\frac{d(y_{k}^{1},p)}{r_{k}^{1}}=\infty \quad \mbox{and} \quad
\limsup_{k\to\infty}\frac{d(u_{k}^{1},p)}{r_{k}^{1}}<\infty ;\end{equation*}
\begin{equation*}(i_3) \qquad \limsup_{k\to\infty}\frac{d(y_{k}^{1},p)}{r_{k}^{1}}=\infty \quad \mbox{and} \quad
\limsup_{k\to\infty}\frac{d(u_{k}^{1},p)}{r_{k}^{1}}=\infty. \end{equation*}
 Reasoning as in the proofs of \eqref{e3.17} and \eqref{e3.22} we can show that $c=0$ if $(i_1)$
 or $(i_2)$ holds. Thus it is sufficient to
 consider only case $(i_3).$ Passing to the subsequence we may suppose that
\begin{equation}\label{e3.24} \lim_{j\to\infty}\frac{d(y_{k_{j}}^{1},p)}{r_{k_{j}}^{1}}=
\lim_{j\to\infty}\frac{d(u_{k_{j}}^{1},p)}{r_{k_{j}}^{1}}=\infty. \end{equation} Indeed,
if there is not a subsequence for which \eqref{e3.24} holds, then we can reduce the
situation to cases $(i_1)$ or $(i_2)$ which were considered above. In addition to
\eqref{e3.24} we may assume that there are $k_1, k_2 \in(0,\infty)$ such that
\begin{equation}\label{e3.25} \lim_{j\to\infty}\frac{d(y_{k_{j}}^{1},p)}{d(u_{k_{j}}^{1},p)}
=k_1 \end{equation} and
\begin{equation}\label{e3.26}
\lim_{j\to\infty}\frac{d(y_{k_{j}}^{1},u_{k_{j}}^{1})}{d(u_{k_{j}}^{1},p)}=k_2
\end{equation}
because if
\begin{equation} \lim_{j\to\infty}\frac{d(y_{k_{j}}^{1},p)}{d(u_{k_{j}}^{1},p)}=0 \quad \mbox{or}
\quad \lim_{j\to\infty}\frac{d(y_{k_{j}}^{1},p)}{d(u_{k_{j}}^{1},p)}=\infty,
\end{equation}
then, similarly \eqref{e3.22}, we find
\begin{equation*} \lim_{k\to\infty}\Theta_{n+3}(x_{k}^{1,0},x_{k}^{1,1},...,x_{k}^{1,n},y_{k}^{1},u_{k}^{1})
=(-1)^{n+3}\begin{vmatrix}
    0&1&1&...&1&1\\
    1&0&0&...&0&1\\
    1&0&0&...&0&1\\
    \vdots&\vdots&\vdots&...&\vdots&\vdots\\
    1&0&0&...&0&1\\
    1&1&1&...&1&0\\
    \end{vmatrix}=0
\end{equation*}
or, respectively,
\begin{equation*} \lim_{k\to\infty}\Theta_{n+3}(x_{k}^{1,0},x_{k}^{1,1},...,x_{k}^{1,n},y_{k}^{1},u_{k}^{1})
=(-1)^{n+3}\begin{vmatrix}
    0&1&1&...&1&1\\
    1&0&0&...&0&1\\
    1&0&0&...&0&1\\
    \vdots&\vdots&\vdots&...&\vdots&\vdots\\
    1&1&1&...&1&0\\
    1&0&0&...&0&1\\
    \end{vmatrix}=0.
\end{equation*}
Limit relations \eqref{e3.24}, \eqref{e3.25} and \eqref{e3.26} imply that the quantity
$$(-1)^{n+3}\lim_{k\to\infty}\Theta_{n+3}(x_{k}^{1,0},x_{k}^{1,1},...,x_{k}^{1,n},y_{k}^{1},u_{k}^{1})$$
equals the determinant of the matrix with the second and third rows of the form
$$(1,0,0,...,0,k_1,1) \quad \mbox{if} \quad k_{1}<1 \quad \mbox{or} \quad(1,0,0,...,0,1,k_1)\quad \mbox{if}
\quad k_{1}\ge 1.$$ Consequently the equality
\begin{equation*}\lim_{k\to\infty}\Theta_{n+3}(x_{k}^{1,0},x_{k}^{1,1},...,x_{k}^{1,n},y_{k}^{1},u_{k}^{1})=c=0 \end{equation*}
holds in all possible cases and \eqref{eq.3.7} follows.

Suppose now that there exist some sequences $\tilde x^{i},\, i=0,...,n,$ with the
properties (i) and (ii). Limit relations \eqref{eq.3.4} imply that the quantities
$$r_{m}:=\mathop{\lor}\limits_{i=0} ^ {n}d(x_{m}^{i},p)$$ become vanishingly small with
$m\to\infty.$ Consequently we can consider $\tilde r=\{r_m\}_{m\in\mathbb N}$ as a
normilizing sequence. As in the proof of the Theorem~\ref{Embed}, going to subsequence
we can assume all $\tilde x^{i}$  and $\tilde p$ to be mutually stable. Let $\tilde
X_{p,\tilde r}$ be a maximal self-stable family such that $\tilde x^{i}\in\tilde
X_{p,\tilde r}$ for $i=0,...,n$ and let $\Omega_{p,\tilde r}^{X}$ be the metric
identification of $\tilde X_{p,\tilde r}.$ Write $$\alpha_{0}=\pi(\tilde x^{0}),
\alpha_{1}=\pi(\tilde x^{1}), ..., \alpha_{n}=\pi(\tilde x^{n})$$ where $\pi$ is the
natural projection of $\tilde X_{p,\tilde r}$ on $\Omega_{p,\tilde r}^{X}.$ Going to the
limit under $m\to\infty$ and using \eqref{eq.3.5} we obtain
\begin{equation}\label{e3.28}
\mathop{\land}\limits_{k=1}^{n}\liminf_{m\to\infty}\Theta_{k+1}(x_{m}^{0},
x_{m}^{1},...,x_{m}^{k})=\mathop{\land}\limits_{k=1}^{n}(-1)^{k+1}D_{k}(\alpha_{0},\alpha_{1},...,
\alpha_{k} )>0.
\end{equation}
Similarly for all $\beta, \gamma \in \Omega_{p,\tilde r}^{X}$ property (ii) implies that
\begin{equation*}
\lim_{m\to\infty}\Theta_{n+2}(x_{m}^{0},x_{m}^{1},...,x_{m}^{n},y_{m})=\lim_{m\to\infty}\Theta_{n+3}(x_{m}^{0},x_{m}^{1},...,x_{m}^{n},y_{m},u_{m})=
\end{equation*}
\begin{equation}\label{e3.29}
(-1)^{n+2}D(\alpha_{0},\alpha_{1},...,\alpha_{n},\beta)=(-1)^{n+3}D_{n+2}(\alpha_{0},\alpha_{1},...,
\alpha_{n},\beta,\gamma)=0,
\end{equation}
where $\{y_m\}_{m\in\mathbb N}\in \tilde X_{p,\tilde r}$ and $\{u_m\}_{m\in\mathbb N}\in
\tilde X_{p,\tilde r}$ such that $\pi(\{y_m\}_{m\in\mathbb N})=\beta$ and
$\pi(\{u_m\}_{m\in\mathbb N})=~\gamma.$ Hence by Theorem~\ref{Blumenthal} the pretangent
space $\Omega_{p,\tilde r}^{X}$ has an isometric embedding in $E^{n}$ but there are not
isometric embeddings of $\Omega_{p,\tilde r}^{X}$ in $E^{l}$ with $l<n,$ as required.
\end{proof}


\medskip
\noindent
{\bf Viktoriia Bilet} \hspace{5cm} {\bf Oleksiy Dovgoshey}\\
IAMM of NASU, Donetsk \hspace{3.4cm} IAMM of NASU, Donetsk \\
E-mail: biletvictoriya@mail.ru \hspace{2.6 cm} E-mail: aleksdov@mail.ru  \\
\end{document}